\documentclass[12pt,reqno]{amsart}
\usepackage{amsmath}
\usepackage{amsfonts}
\usepackage{amssymb}
\usepackage{caption}
\usepackage{latexsym}
\usepackage{color}
\usepackage{hyperref}
\usepackage{url}
\usepackage{enumerate}
\usepackage{multicol} %用于实现在同一页中实现不同的分栏

\renewcommand{\proof}{\noindent{\it Proof.\ \ }}
\renewcommand{\qed}{\ifmmode\square\else\nolinebreak\hfill
$\Box$\fi\par\vskip12pt}

\renewcommand\a{\alpha}    

\newcommand\Ga{\mathrm{\Gamma}}   
\newcommand\Sig{{\it \Sigma}}

\newcommand\A{\mathrm{A}}    \newcommand\C{\mathbf{C}}   \newcommand\D{\mathrm{D}}
       
\newcommand\K{\mathsf{K}}  \newcommand\M{\mathrm{M}}

\newcommand\ZZ{\mathbb{Z}}     \newcommand\BB{\mathcal{B}}

          \newcommand\Aut{\mathrm{Aut}}
   \newcommand\Cay{\mathrm{Cay}}      
      \newcommand\Inn{\mathrm{Inn}}         
            
    \newcommand\Sy{\mathrm{S}}         \newcommand\Sym{\mathrm{Sym}}
          
  \newcommand\ov{\overline}          \newcommand\la{\langle}
\newcommand\ra{\rangle}

\newcommand\AGL{\mathrm{AGL}}                  
\newcommand\ASigmaL{\mathrm{A \Sigma L}}       \newcommand\AGammaL{\mathrm{A\Gamma L}}

\newcommand\GammaL{\mathrm{\Gamma L}}          
\newcommand\GL{\mathrm{GL}}                    
                    
\newcommand\PGL{\mathrm{PGL}}                  
                  
\newcommand\PSL{\mathrm{PSL}}

\newtheorem{theorem}{Theorem}[section]%
\newtheorem{lemma}[theorem]{Lemma}%
\newtheorem{corollary}[theorem]{Corollary}%
\newtheorem{proposition}[theorem]{Proposition}%
\newtheorem{example}[theorem]{Example}%
\newtheorem{hypothesis}[theorem]{Hypothesis}%

\begin{document}

\title[On arc-transitive inner-automorphic Cayley graphs on dihedral groups]
{On arc-transitive inner-automorphic Cayley graphs on dihedral groups}

\thanks{$^*$Corresponding author. }
\thanks{2020 MR Subject Classification 20B15, 20B30, 05C25.}

\author[J.-J. Huang, J.-H. Xie]{Jun-Jie Huang, Jin-Hua Xie$^*$}
\address{Jun-Jie Huang\\
School of Mathematical Sciences, Laboratory of Mathematics
and Complex Systems, MOE\\
Beijing Normal University\\
Beijing \\
100875, P. R. China}

\address{Jin-Hua Xie\\
Center for Combinatorics, LPMC \\
Nankai University\\
Tianjin \\
300071 , P. R. China}

{\email{jjhuang@bnu.edu.cn(J.-J. Huang) jinhuaxie@nankai.edu.cn (J.-H. Xie)}
\maketitle

\begin{abstract}
A Cayley graph $\Cay(G,S)$ is said to be inner-automorphic if $S$ is a union of conjugacy classes of a group $G$, and arc-transitive if its full automorphism group acts transitively on the set of arcs. In this paper, we characterize four well-known families of arc-transitive graphs that arise as connected inner-automorphic Cayley graphs on dihedral groups, and we provide a necessary condition for other connected arc-transitive Cayley graphs on dihedral groups to be inner-automorphic. We further construct an infinite family of examples satisfying this condition, thereby demonstrating the existence of such graphs. Finally, we complete the classification of all 2-distance-transitive connected inner-automorphic Cayley graphs on dihedral groups.
\end{abstract}

\qquad {\textsc k}{\scriptsize \textsc {eywords.} arc-transitive, inner-automorphic Cayley graph, dihedral group, automorphism group} {\footnotesize}

\section{Introduction}
In this paper, all graphs are finite, simple, connected and undirected. For a graph $\Ga$, we denote by $V(\Ga)$, $E(\Ga)$ and $\Aut(\Ga)$ its vertex set, edge set, and full automorphism group, respectively. Let $G$ be a finite group. For every element $g\in G$, we define the permutations $R(g): x \mapsto xg$ and $L(g): x \mapsto g^{-1}x$ on the set $G$. The sets
$R(G):= \{R(g) \mid g \in G\}$ and $L(G):= \{L(g) \mid g \in G\}$ form regular subgroups of the symmetric group $S_G$, known respectively as the {\em right regular representation} and {\em left regular representation} of $G$. A fundamental relation between them is given by $L(G)R(G) = R(G)\Inn(G)$, where $\Inn(G)$ denotes the inner automorphism group of $G$.

Let $G$ be a finite group and $S \subseteq G \setminus \{1\}$ an inverse-closed subset of $G$ (that is, $S = S^{-1}:=\{s^{-1}\mid s\in S\}$). The {\em Cayley graph} of $G$ with respect to $S$, denoted by $\Cay(G,S)$, has vertex set $G$, and two vertices $x,y\in G$ are adjacent if and only if $y x^{-1} \in S$. Clearly, $\Cay(G,S)$ is connected if and only if $\langle S\rangle =G$. Let $\Ga=\Cay(G,S)$. It is easy to verify that the right regular representation $R(G)$ forms a regular subgroup of $\Aut(\Ga)$, and that $\Aut(G,S):=\{\a\in \Aut(G) \mid S^\a=S\}$ is a subgroup of $\Aut(\Ga)_1$, the stabilizer of the identity element $1$ in $\Aut(\Ga)$. In contrast, the left regular representation $L(G)$ is generally not contained in $\Aut(\Ga)$. A fundamental result of Godsil \cite{Godsil} shows that the normalizer of $R(G)$ in $\Aut(\Ga)$ is the semidirect product $R(G) \rtimes \Aut(G,S)$. The Cayley graph was initially introduced by A. Cayley in 1878 as a graphical representation of abstract groups. Research on Cayley graphs spans a variety of challenging topics, including the graph isomorphism problem \cite{FK,LLP07,XFZ}, graphical regular representations (GRR) of groups \cite{Xia}, symmetry and normality of Cayley graphs \cite{LP,Mar03,YFZC}, as well as the construction and classification of specific families of Cayley graphs \cite{DMM,HFZ,HFZY2025,Pan14}.

When the set $S$ is a union of conjugacy classes of $G$, the Cayley graph $\Cay(G,S)$ is said to be {\em inner-automorphic}, a notion named by Cameron~\cite{Cameron}. Clearly, $\Cay(G,S)$ is inner-automorphic if and only if $\Inn(G)\leq \Aut(\Cay(G,S))$, which in turn is equivalent to $L(G)\leq \Aut(\Cay(G,S))$, since $L(G)R(G)=R(G)\Inn(G)$. Additional equivalent characterizations can be found in Proposition~\ref{dual}. In the literature, inner-automorphic Cayley graphs have been widely studied under various names, such as ``normal Cayley graph" in \cite{LXZ2024, Roichman}, ``group graph" in \cite{Fisk}, ``conjugacy class graph" in \cite{Ito},
``quasi-abelian Cayley graph" in \cite{WX,Zgr01}, ``holomorph Cayley graph" in \cite{Li06}, ``central Cayley graph" in \cite{GGRV,PV}, and ``dual Cayley graph" in \cite{Pan}.

The study of inner-automorphic Cayley graphs dates back to Imrich's work \cite{Imrich} in 1979, which established that for $n\geq 5$, an inner-automorphic Cayley graph $\Cay(\Sy_n,S)$ has greatest possible connectivity when the set $S$ contains at least one odd permutation in the symmetric group $\Sy_n$ of degree $n$. In subsequent years, this family of graphs has continued to attract considerable interest. A number of papers have since explored a wide range of problems related to these graphs, including Hamiltonicity (see \cite{WX}), the isomorphism problem (see \cite{PV}), the property of being $c$-expanders \cite{Roichman96,Roichman}, spectral properties and connectivity \cite{Imrich,Ito,Zieschang}, the second largest eigenvalue \cite{LXZ2024}, the ultimate independence ratio and hom-idempotence \cite{HHP,LLT}, as well as symmetry properties \cite{HFYK,Pan,Zgr02}. In this paper, we focus on the study of symmetry in inner-automorphic Cayley graphs.

Now we introduce some notation. Let $u$ and $v$ be two distinct vertices of a graph $\Ga$. The {\em distance} between $u$ and $v$ in $\Ga$ is the smallest length of paths between $u$ and $v$, denoted by $d_\Ga(u,v)$. For a positive integer $i$, denote by $\Ga_i(u)$ the set of vertices at distance $i$ with vertex $u$ in $\Ga$. In particular, the neighborhood $\Ga_1(u)$ is abbreviated as $\Ga(u)$. For a positive integer $s$, an {\em $s$-arc} of $\Ga$ is a sequence of $s+1$ vertices $(u_0,u_1,\ldots,u_s)$ in $\Ga$ such that $u_i$ is adjacent to $u_{i+1}$ for $0\leq i\leq s-1$ and $u_{j-1}\neq u_{j+1}$ for $1\leq j\leq s-1$. A 1-arc is simply called an {\em arc}. Let $X$ be a transitive subgroup of  $\Aut(\Gamma)$. Then $\Ga$ is called {\em $(X,s)$-arc-transitive} if $X$ is transitive on the set of $s$-arcs of $\Ga$, and moreover, it is called {\em $(X,s)$-distance-transitive}, if for each vertex $u\in V(\Ga)$, the stabilizer $X_u$ is transitive on $\Ga_i(u)$ for all $i\leq s$. When $X=\Aut(\Ga)$, we simply say that $\Ga$ is {\em $s$-arc-transitive} or {\em $s$-distance-transitive}, respectively. Furthermore, a $1$-arc-transitive graph is called an {\em arc-transitive} or {\em symmetric} graph.

The symmetry of inner-automorphic Cayley graphs has emerged as an important research theme in recent algebraic graph theory.
For example, Li~\cite{Li06} showed that every connected $(X,2)$-arc-transitive inner-automorphic Cayley graph $\Cay(G,S)$ with $X=R(G)L(G)$ is isomorphic to the complete bipartite graph $\K_{3,3}$. The classification of $2$-arc-transitive inner-automorphic Cayley graphs on nonabelian simple groups were given in~\cite{Pan}. More recently, the first author and coauthors classified all $3$-arc-transitive inner-automorphic Cayley graphs, and investigated the properties of $2$-distance-transitive inner-automorphic Cayley graphs on nonabelian simple groups, refer to~\cite{HFYK}.

Building on existing results, we investigate the arc-transitivity of inner-automorphic Cayley graphs on dihedral groups. Considerable progress has been made in classifying {\em dihedrants} (a Cayley graph on a dihedral group) with specific symmetries. For example, the distance-transitive and locally primitive cases were classified in~\cite{MP06} and~\cite{Pan14}, respectively. The classification of 2-arc-transitive dihedrants was completed through a sequence of papers~\cite{DMM,Jin23,QDK}. More recently, Huang et al.~\cite{HFZY2025} provided a full classification of connected 2-distance-transitive dihedrants, showing that each such graph is either a complete multipartite graph or a known 2-arc-transitive example. Despite these achievements, a complete classification of arc-transitive dihedrants remains an open and challenging problem. This has motivated the study of families satisfying additional constraints. For example, Kov\'{a}cs~\cite{Kov} classified arc-transitive Cayley graphs on dihedral groups $\D_{2p^n}$ for a prime $p$; arc-regular actions on dihedrants were examined in~\cite{KMM}; and arc-transitive dihedrants of small valency have been studied in~\cite{DFKX,KO,WX06}.

This paper is devoted to the study of arc-transitive inner-automorphic Cayley graphs on dihedral groups. Our first theorem shows that, with one exceptional case, all such graphs are known.

\begin{theorem}\label{Thm:dihedrant}
Let $\Ga=\Cay(G,S)$ be a connected arc-transitive inner-automorphic Cayley graph, where $G=\la a,b\mid a^n=b^2=1,a^b=a^{-1}\ra\cong\D_{2n}$ with $n\geq 2$. Then one of the following holds:
\begin{enumerate}[\rm (i)]
  \item $\Ga\cong\K_{n,n}$, and either $S=\{a^ib\mid 1\leq i\leq n\}$, or $n=2k$ is even and $S=\{a^{2i}b, a^{2i+1}\mid 1\leq i\leq k\}$ or $S=\{a^{2i+1}b, a^{2i+1}\mid 1\leq i\leq k\}$;
  \item $\Ga\cong\K_{n,n}-n\K_2$ with $n=2k$ even and $k\geq3$ odd, and either $S=\{a^{2i}b, a^{2j+1}\mid 1\leq i,j\leq k,2j+1\neq k\}$ or $S=\{a^{2i+1}b, a^{2j+1}\mid 1\leq i,j\leq k,2j+1\neq k\}$;
  \item $\Ga\cong\K_{2n}$ and $S=G\setminus\{1\}$;
  \item $\Ga\cong\K_{m[t]}$ and $S=G\setminus H$ and $H$ is the unique subgroup of $\la a\ra$ of order $t$, where $m\geq 3$, $t\geq 2$ and $mt=2n$;
  \item $\Ga$ is a bipartite graph with biparts $\la a^2,b\ra$ and $\la a^2,b\ra a$, of girth $4$ and diameter $3$,
        and $(a^\pi b)^G\subset S\subseteq (a^\pi b)^G\cup\la a^2\ra a$ with $n=2k$, $\pi\in\{0,1\}$ and $0<|\la a^2\ra a\cap S|\leq k-2$.
\end{enumerate}
\end{theorem}

It is well-known that the complete bipartite graph $\K_{n,n}$, with $n\geq3$ is 3-arc-transitive but not 4-arc-transitive; the graph $\K_{n,n}-n\K_2$ (a complete bipartite graph minus a matching) with $n\geq4$, as well as the complete graph $\K_{2n}$, are 2-arc-transitive but not 3-arc-transitive; and the complete multipartite graph $\K_{m[t]}$} with $m\geq3$ and $t\geq 2$ is 2-distance-transitive but not 2-arc-transitive. Combining these facts with Theorem~\ref{Thm:dihedrant}, we immediately obtain the following corollary.

\begin{corollary}\label{dual:s-tran}
Let $n\geq3$ be an integer and let $\Ga=\Cay(G,S)$ be a connected inner-automorphic Cayley graph on $G\cong\D_{2n}$. Then the following statements hold:
\begin{enumerate}[\rm (i)]
  \item $\Ga$ is at most $3$-arc-transitive;
  \item $\Ga$ is $3$-arc-transitive if and only if $\Ga\cong\K_{n,n}$;
  \item $\Ga$ is $2$-arc-transitive if and only if it is isomorphic to $\Ga\cong\K_{2n}$, $\K_{n,n}$, or $\K_{n,n}-n\K_2$ with $n\geq4$.
\end{enumerate}
\end{corollary}

For a $2$-distance-transitive inner-automorphic Cayley graph on $G \cong \D_{2n}$, case (v) of Theorem~\ref{Thm:dihedrant} is excluded by the following result.

\begin{corollary} \label{class:2distran}
Let $\Ga=\Cay(G,S)$ be a connected inner-automorphic Cayley graph on $G\cong\D_{2n}$ with $n\geq2$.
Then $\Ga$ is $2$-distance-transitive if and only if $\Ga$ is isomorphic to one of the graphs listed in cases $(i)$--$(iv)$ of Theorem~$\ref{Thm:dihedrant}$.
\end{corollary}

Our second theorem provides a family of arc-transitive inner-automorphic Cayley graphs that satisfies case (v) of Theorem~\ref{Thm:dihedrant}.

\begin{theorem}\label{Thm:dualCay}
Let $p$ be an odd prime and let $G=\la a,b\mid a^{4p}=b^2=1,a^b=a^{-1}\ra\cong\D_{8p}$.
Take $S_\pi=(a^\pi b)^G\cup \{a^i\mid (i,4p)=1\}$ with $\pi\in \{0,1\}$.
Then $\Cay(G,S_\pi)$ is a connected arc-transitive inner-automorphic Cayley graph. Moreover,  $\Aut(\Cay(G,S_\pi))\cong\ZZ_2^{4p}.(\Sy_{2p}\times\ZZ_2)$.
\end{theorem}

We would like to emphasize that there exist arc-transitive inner-automorphic Cayley graphs on dihedral groups
that are not isomorphic to those in Theorem~\ref{Thm:dualCay}, see Examples \ref{D60} and~\ref{D84}.

The paper is organized as follows. Section~\ref{Preliminaries} contains basic definitions and preliminary results on groups and graphs that will be used throughout. The proofs of Theorems~\ref{Thm:dihedrant} and~\ref{Thm:dualCay} are given in Sections~\ref{sec:proof} and~\ref{sec:structure}, respectively.

\section{Preliminaries}\label{Preliminaries}

This section lists definitions and properties concerning groups and graphs used throughout the paper. For a positive integer $n$, let $\ZZ_n$ denote the cyclic group of order $n$, and $\D_{2n}$ the dihedral group of order $2n$. For a prime $p$ and a positive integer $r$, denote by $\ZZ_p^r$ the elementary abelian group of order $p^r$. For two groups $A$ and $B$, denote by $A\times B$ the direct product of $A$ and $B$, by $A\rtimes B$ a semidirect product of $A$ by $B$, and by $A.B$ an extension of $A$ by $B$. For a group $G$ and an element $g\in G$, the conjugacy class of $g$ in $G$ is denoted by $g^G$.

In graph theory, we denote by $\C_n$ the cycle graph of length $n$, by $\K_n$ the complete graph of order $n$, by $\K_{n,n}$ the complete bipartite graph of order $2n$, by $\K_{n,n}-n\K_2$ the subgraph of $\K_{n,n}$ minus a matching, and by $\K_{m[t]}$ the complete multipartite graph consisting $m\geq 3$ parts of size $t\geq 2$. For a graph $\Ga$ with a nonempty subset $U\subseteq V(\Ga)$, denote by $[U]$ the subgraph induced by $U$. Let $W\subseteq V(\Ga)$ be such that $U\cap W=\emptyset$.
Denote by $[U,W]$ the bipartite subgraph with vertex set $U\cup W$ and edge set $\{\{u,v\}\in E(\Ga)\mid u\in U, v\in W\}$.

The following proposition gives necessary and sufficient conditions for a Cayley graph to be inner-automorphic,
see \cite[Theorem 2.2 and Lemma 2.4]{Pan}.

\begin{proposition}\label{dual}
Let $\Ga=\Cay(G,S)$ be a connected Cayley graph. Then the following are equivalent:
\begin{enumerate}[\rm (i)]
  \item $\Ga$ is an inner-automorphic Cayley graph;
  \item $\Inn(G)\leq\Aut(G,S)$;
   \item $L(G)\leq\Aut(\Ga)$;
  \item $\tau\in\Aut(\Ga)$, where $\tau:x\mapsto x^{-1}$ for all $x\in G$;
  \item $S=\{s_1,s_1^{-1}\}^G\cup \cdots\cup\{s_m,s_m^{-1}\}^G$ for some $s_1,\ldots, s_m\in G\setminus\{1\}$.
\end{enumerate}
\end{proposition}

This equivalence explains the varying definitions of inner-automorphic Cayley graphs found in the literature~\cite{Fisk,Ito,Li06,PV,Roichman,WX}.

\medskip
Let $\Ga=\Cay(G,S)$ be a Cayley graph. Then $\Ga$ is called {\em normal}, if $R(G)$ is a normal subgroup of $\Aut(\Ga)$. It is known from~\cite{Godsil} that $N_{\Aut(\Ga)}(R(G))=R(G)\rtimes\Aut(G,S)$. Consequently, $\Ga$ is a normal Cayley graph if and only if $\Aut(\Ga)_1=\Aut(G,S)$. On the one hand, the map $\tau$ defined in Proposition~\ref{dual} is an automorphism of $G$ if and only if $G$ is abelian. Combining these two facts yields the following result.

\begin{proposition} \label{normal}
Let $\Ga=\Cay(G,S)$ be a normal inner-automorphic Cayley graph on a finite group $G$. Then $G$ is abelian.
\end{proposition}

The following elementary but fundamental description of the conjugacy classes in dihedral groups will be used repeatedly and is established here for convenience.

\begin{proposition} \label{conju-D2n}
Let $G=\la a,b\mid a^n=b^2=1,a^b=a^{-1}\ra\cong\D_{2n}$ with $n\geq2$. The conjugacy classes of the non-identity elements of $G$ are given as follows.
\begin{enumerate}[\rm (i)]
  \item For $n$ odd, we have $(ab)^G=\{a^ib\mid 1\leq i\leq n\}$ and $(a^i)^G=\{a^{i},a^{-i}\}$ with $1\leq i\leq (n-1)/2$.
  \item For $n=2k$ even, we have $(ab)^G=\{a^{2i+1}b\mid 1\leq i\leq k\}$, $b^G=\{a^{2i}b\mid 1\leq i\leq k\}$, $(a^k)^G=\{a^{k}\}$ and $(a^i)^G=\{a^{i},a^{-i}\}$ with $1\leq i\leq k-1$.
\end{enumerate}
\end{proposition}

We now describe the automorphism group of dihedral groups. Let $G=\la a,b\mid a^n=b^2=1, a^b=a^{-1}\ra\cong\D_{2n}$.
Define the following maps on the generators of $G$:
\begin{align*}
&\theta_{a^i}:a\mapsto a,~b\mapsto a^ib, \text{~where~} 1\leq i\leq n;\\
&\tau_{a^j}:a\mapsto a^j,~b\mapsto b,\text{~where~} 1\leq j\leq n \text{~and~} (j,n)=1.
\end{align*}
These maps extend to automorphisms of $G$, which we denote by the same symbols $\theta_{a^i}$ and $\tau_{a^j}$. It is well-established that the full automorphism group $\Aut(G)$ is generated by the set $\{\theta_{a^i},\tau_{a^j}\}$, also refer to~\cite[Section 3]{XFZ}.

\begin{proposition} \label{Aut:D2n}
Let $G=\la a,b\mid a^n=b^2=1, a^b=a^{-1}\ra\cong\D_{2n}$ with $n\geq2$.
Then $\Aut(G)=\la \theta_{a^i},\tau_{a^j}\mid 1\leq i,j\leq n, (j,n)=1\ra=\la\theta_{a}\ra\rtimes\la \tau_{a^j}\mid 1\leq j\leq n, (j,n)=1\ra$.
\end{proposition}

Let $G$ be a transitive permutation group on a set $\Omega$. A non-empty subset $\Delta\subseteq \Omega$ is called a {\em block} of $G$, if for every $g\in G$, either $\Delta^g=\Delta$ or $\Delta^g\cap \Delta=\emptyset$.
A block $\Delta$ is {\em trivial} if $|\Delta|=1$ or $|\Delta|=|\Omega|$. The group $G$ is called {\em primitive} if it has only trivial blocks in $\Omega$, and {\em quasiprimitive} if every non-trivial normal subgroup of $G$ is transitive on $\Omega$. Clearly, every primitive group is quasiprimitive, but the converse is not true in general.

\begin{table}[!htb]
\caption{Quasiprimitive permutation groups with a regular dihedral subgroup}
\label{d-group}
\begin{center}
\begin{tabular}{l|l|l|l} \hline
  $G$ & $G_\a$ & $H$ & Conditions \\ \hline
  $\A_4$ & $\ZZ_3$ & $\D_4$ &   \\
  $\Sy_4$ & $\Sy_3$ & $\D_4$ &   \\
  $\AGL(3,2)$ & $\GL(3,2)$ & $\D_8$ &   \\
  $\AGL(4,2)$ & $\GL(4,2)$ & $\D_{16}$ &    \\
  $\ZZ_2^4\A_7$ & $\A_7$ & $\D_{16}$ &   \\
  $\ZZ_2^4\rtimes\Sy_6$ & $\Sy_6$ & $\D_{16}$ &   \\
  $\ZZ_2^4\rtimes\A_6$ & $\A_6$ & $\D_{16}$ &   \\
  $\ZZ_2^4\rtimes\Sy_5$ & $\Sy_5$ & $\D_{16}$ &  \\
  $\ZZ_2^4\rtimes\GammaL(2,4)$ & $\GammaL(2,4)$ & $\D_{16}$ &   \\
  $\M_{12}$ & $\M_{11}$ & $\D_{12}$ &  \\
  $\M_{22}.\ZZ_2$ & $\PSL(3,4).\ZZ_2$ & $\D_{22}$ &   \\
  $\M_{24}$ & $\M_{23}$ & $\D_{24}$ &  \\
  $\Sy_{2n}$ & $\Sy_{2n-1}$ & $\D_{2n}$ &    \\
  $\A_{4n}$ & $\A_{4n-1}$ & $\D_{4n}$ &    \\
  $\PSL(2,r^f).o$ & $\ZZ_r^f\rtimes\ZZ_{\frac{r^f-1}{2}}.o$ & $\D_{r^f+1}$ & $r^f\equiv3\pmod{4}$,~ $o\leq\ZZ_2\times\ZZ_f$  \\
  $\PGL(2,r^f).\ZZ_e$ & $\ZZ_r^f\rtimes\ZZ_{r^f-1}.\ZZ_e$ & $\D_{r^f+1}$ & $r^f\equiv1\pmod{4},~e\mid f$ \\ \hline
\end{tabular}
\end{center}
\end{table}
The quasiprimitive permutation groups that contain a regular dihedral subgroup were classified in \cite[Theorem 1.5]{Li06}, although two examples were unfortunately missed: the affine groups $\AGammaL(2,4)\cong\ZZ_2^4\rtimes\GammaL(2,4)$ and $\ASigmaL(2,4)\cong\ZZ_2^4\rtimes\Sy_5$ of degree $2^4$. A correct and complete classification was later provided in \cite[Theorem 3.3]{SLZ}.

\begin{proposition}\label{Qd-group}
Let $G$ be a quasiprimitive permutation group on a set $\Omega$ such that $G$ contains a regular dihedral subgroup $H$.
Let $\a\in\Omega$. Then $G$ is $2$-transitive on $\Omega$ and $(G,G_\a,H)$ is listed in Table~$\ref{d-group}$.
\end{proposition}

Finally, we introduce some concepts related to $X$-vertex transitive graphs $\Ga$, where $X\leq \Aut(\Ga)$.  Let $\BB=\{B_1,\ldots,B_n\}$ be a $X$-invariant partition of $V(\Ga)$, namely, for each $x\in X$ and $B_i\in\BB$, either $B_i^x=B_i$ or $B_i^x\cap B_i=\emptyset$. Then the {\em quotient graph} $\Ga_\BB$ of $\Ga$ induced on $\BB$ is the graph with vertex set $\BB$, and $B_i$ is adjacent to $B_j$ in $\Ga_\BB$ if there exist $u\in B_i$ and $v\in B_j$ such that $\{u,v\}$ is an edge of $\Ga$. The graph $\Ga$ is said to be an {\em $r$-cover} of $\Ga_\BB$ if for each edge $\{B_i,B_j\}$ of $\Ga_\BB$ and each vertex $u\in B_i$, we have $|\Ga(u)\cap B_j|=r$. In particular, if $r=1$, then $\Ga$ is called a {\em cover} of $\Ga_\BB$. Moreover, if $\BB$ is the set of orbits of an intransitive normal subgroup $N$ of $X$, then $\Ga$ is called a {\em normal $r$-cover} of $\Ga_\BB$, and we write $\Ga_\BB=\Ga_N$.

\section{Proof of Theorem~\ref{Thm:dihedrant}}\label{sec:proof}

For each finite group $G$ of order $n$, we know that $\Cay(G,G\setminus\{1\})\cong\K_n$ is an inner-automorphic Cayley graph on $G$. In~\cite[Lemma 3.3]{HFYK}, it was shown that the cycle $\C_n$ is an inner-automorphic Cayley graph $\Cay(G,S)$ on $G$ if and only if either $G=\la a\ra\cong\ZZ_n$ and $S=\{a,a^{-1}\}$, or $G=\la a,b\ra\cong\D_4$ and $S=\{a,b\}$.

In what follows, we examine whether the three graphs $\K_{n,n}$, $\K_{n,n} - n\K_2$, and $\K_{m[t]}$ are inner-automorphic Cayley graphs. We first provide a necessary and sufficient condition for a Cayley graph to be isomorphic to $\K_{n,n}$ or $\K_{m[t]}$.

\begin{lemma} \label{arctran:kmb}
Let $\Ga=\Cay(G,S)$ be a connected Cayley graph. Then the following statements hold.
\begin{enumerate}[\rm (i)]
  \item $\Ga\cong\K_{n,n}$ if and only if $S=G\setminus H$ and $H$ is a subgroup of $G$ of index $2$. In particular, $\Ga$ is an inner-automorphic Cayley graph.
  \item $\Ga\cong\K_{m[t]}$ with $m\geq3$ and $t\geq2$ if and only if $S=G\setminus H$, and $H$ is a subgroup of order $t$, and $m=|G:H|$.
\end{enumerate}
\end{lemma}

\proof Suppose that $\Ga\cong \K_{n,n}$ or $\K_{m[t]}$. Then $V(\Ga)=\{1\}\cup S\cup \Ga_2(1)$. Let $H=\{1\}\cup\Ga_2(1)$. Then $S=G\setminus H$ and the subgraph $[H]$ is an edgeless graph. Take two arbitrary elements $x,y\in H$. Clearly, $xy^{-1}\in S\cup H$, and since $S=S^{-1}$, we have $y^{-1}\in H$. If $xy^{-1}\in S$, then $x\neq y$ and $\{x,y\}$ is an edge of the subgraph $[H]$, which is impossible. Thus, $xy^{-1}\in H$, and so $H$ is a subgroup of $G$. If $\Ga\cong\K_{n,n}$, then $|\Ga_2(1)|=n-1$, and so $|H|=n$, as the necessity of part (i). If $\Ga\cong\K_{m[t]}$, then $|\Ga_2(1)|=t-1$, and so $|H|=t\geq2$ and $m=|G:H|\geq3$. This proves the necessity of part (ii).

Conversely, suppose first that $S=G\setminus H$ and the other conditions of $H$ holds. Let $m=|G:H|$. Then we may set $G=H\cup Hx_1\cup\cdots\cup Hx_{m-1}$. Note from $S=G\setminus H$ that $[H]$ is an edgeless graph. Since $R(G)\leq\Aut(\Ga)$, we drive that $R(x_i)$ maps the subgraph $[H]$ to $[Hx_i]$ for every $i\in\{1,\ldots,m-1\}$, which implies that $[H]\cong [Hx_i]$. Moreover, for each $h_1\in H$ and $h_2x_i\in Hx_i$ with $1\leq i\leq m-1$, we have $h_2x_ih_1^{-1}\notin H$, and hence $h_2x_ih_1^{-1}\in S$. It follows that $h_1$ is adjacent to $h_2x_i$, and so the subgraph $[H,Hx_i]\cong\K_{t,t}$. Therefore, if $m=2$, then $\Ga\cong\K_{n,n}$, and if $m\geq3$, then $\Ga\cong\K_{m[t]}$. This proves the sufficiency of part (i) and (ii).

Accordingly, when $\Ga\cong\K_{n,n}$, the condition $|G:H|=2$ directly implies that $\Ga$ is inner-automorphic, completing the proof.
\qed

From Lemma~\ref{arctran:kmb}, we have the following corollary.

\begin{corollary} \label{dual:kmb}
Let $\Ga=\Cay(G,S)$ be a connected Cayley graph. Then $\Ga\cong\K_{m[t]}$, with $m\geq3$ and $t\geq2$, is an inner-automorphic Cayley graph if and only if $S=G\setminus H$, and $H$ is a normal subgroup of $G$ order $t\geq2$ and $m=|G:H|\geq3$.
\end{corollary}

\proof For the sufficiency, Lemma~\ref{arctran:kmb}(ii) implies that $\Ga\cong\K_{m[t]}$. Moreover, since $H\unlhd G$, it is clear that $\Ga$ is inner-automorphic.

For the necessity, Lemma~\ref{arctran:kmb}(ii) implies that it suffices to show $H\unlhd G$. This normality follows from the inner-automorphy of $\Ga$ and Proposition~\ref{dual}, which states that $S=G\setminus H$ must be the union of conjugacy classes of some elements in $G$. \qed

Throughout the rest of this section, we always suppose that
\begin{equation*}
G=\la a,b\mid a^n=b^2=1,a^b=a^{-1}\ra\cong\D_{2n} \text{~with~} n\geq2.
\end{equation*}
Utilizing Lemma~\ref{arctran:kmb}, we now establish a criterion for the graph $\K_{n,n}-n\K_2$, with $n\geq4$,
to be equivalent to an inner-automorphic Cayley graph over a dihedral group.

\begin{lemma} \label{dual:knn-nk2}
Let $\Ga=\Cay(G,S)\cong\K_{n,n}-n\K_2$, where $n\geq4$. Then $\Ga$ is an inner-automorphic Cayley graph on $G$ if and only if $n=2k$ with $k\geq3$ odd, and either $S=\{a^{2i}b, a^{2j+1}\mid 1\leq i,j\leq k,2j+1\neq k\}$ or $S=\{a^{2i+1}b, a^{2j+1}\mid 1\leq i,j\leq k,2j+1\neq k\}$.
\end{lemma}

\proof Sufficiency follows readily. To prove the necessity, assume that $\Ga\cong\K_{n,n}-n\K_2$ is an inner-automorphic Cayley graph on a dihedral group $G$. Note that $\Ga$ is a subgraph of $\K_{n,n}$ minus a matching. Write $\Sig=\Cay(G,T)\cong\K_{n,n}$.
Then we may let $\Ga=\Cay(G,S)$ with $S=T\setminus\{x\}$ for some $x\in T$. Clearly, $|S|=|T|-1$. Since $\Ga$ is inner-automorphic, Proposition~\ref{dual} means that $S$ is the union of conjugacy classes of some elements in $G$, which implies that $S^G=S$ and $x^G\cap S=\emptyset$. Noting from Lemma~\ref{arctran:kmb}(i) that $\Sig=\Cay(G,T)\cong\K_{n,n}$ is inner-automorphism, then it follows that $T^G=T$. Thus, we have that
\[
|T|-1=|S|=|S^G|=|(T\setminus\{x\})^G|=|T^G|-|x^G|=|T|-|x^G|,
\]
and so $|x^G|=1$. This implies that $x$ lies in the center of $G$. By Proposition~\ref{conju-D2n}, we see that $n=2k$ is even and $x=a^k$, where $k\geq 2$ as $n\geq 4$. Again by Lemma~\ref{arctran:kmb}(i), we obtain that $T=G\setminus H$ with $H$ a subgroup of $G$ with index $2$. Clearly, $x=a^k\notin H$ as $x\in T$. Notice that all subgroups of $G$ with index $2$ are: $\la a\ra$, $\la a^2,b\ra$ and $\la a^2,ab\ra$. Then $H=\la a^2,b\ra$ or $\la a^2,ab\ra$, and $k\geq3$ is odd. Hence we derive from $S=T\setminus\{x\}$ and $T=G\setminus H$ that
\[
S=\{a^{2i}b, a^{2j+1}\mid 1\leq i,j\leq k,2j+1\neq k\}
\]
or
\[
S=\{a^{2i+1}b, a^{2j+1}\mid 1\leq i,j\leq k,2j+1\neq k\},\]
completing the proof.
\qed

We now characterize connected arc-transitive inner-automorphic Cayley graphs over dihedral groups.

\begin{lemma} \label{arctran:dualredu}
Let $\Ga=\Cay(G,S)$ be a connected arc-transitive inner-automorphic Cayley graph on $G\cong\D_{2n}$ with $n\geq2$.
Then one of the following is exactly true:
\begin{enumerate}[\rm (i)]
    \item $\Ga\cong\K_{n,n}$, and either $S=\{a^ib\mid 1\leq i\leq n\}$, or $n=2k$ is even and $S=\{a^{2i}b, a^{2i+1}\mid 1\leq i\leq k\}$ or $S=\{a^{2i+1}b, a^{2i+1}\mid 1\leq i\leq k\}$;
  \item $\Ga\cong\K_{n,n}-n\K_2$ with $n=2k$ even and $k\geq3$ odd, and either $S=\{a^{2i}b, a^{2j+1}\mid 1\leq i,j\leq k,2j+1\neq k\}$ or $S=\{a^{2i+1}b, a^{2j+1}\mid 1\leq i,j\leq k,2j+1\neq k\}$;
  \item $\Ga\cong\K_{2n}$ and $S=G\setminus\{1\}$;
  \item $\Ga\cong\K_{m[t]}$, $S=G\setminus H$ and $H$ is the unique subgroup of $\la a\ra$ of order $t$, where $m\geq 3$, $t\geq 2$ and $mt=2n$;
  \item $(a^\pi b)^G\subseteq S\subseteq(a^\pi b)^G\cup\la a^2\ra a$ and $0<|\la a^2\ra a\cap S|\leq k-2$, where $n=2k$ and $\pi\in\{0,1\}$.
\end{enumerate}
\end{lemma}

\proof Since $\Ga$ is connected, we have $\la S\ra=G$, and so $a^ib\in S$ for some $1\leq i\leq n$.
By Proposition~\ref{dual}, we have $(a^ib)^G\subseteq S$, and so either $ab\in S$ or $b\in S$.
Let $A=\Aut(\Ga)$ and let $A_1$ be the stabilizer of $A$.
By Proposition~\ref{dual}, we obtain $\tau\in A_1$, where $\tau:x\mapsto x^{-1}$ for all $x\in G$.
Next, we process the proof by considering the two cases.

\medskip
\noindent{\bf Case 1:} Assume that $(ab)^G\cup b^G\subseteq S$.
\medskip

If $S=(ab)^G\cup b^G=G\setminus\la a\ra$, then $\Ga\cong\K_{n,n}$ by Lemma~\ref{arctran:kmb}, as part (i).
We now assume that $(ab)^G\cup b^G\subset S$.
By Propositions~\ref{dual} and~\ref{conju-D2n}, we may assume that
\begin{equation*}
S=(ab)^G\cup b^G\cup\Delta, \text{~where~} \Delta=\{a^{i_1},a^{-i_1},\ldots,a^{i_t}, a^{-i_t}\}\text{~and~}1\leq t\leq [n/2].
\end{equation*}
Then
\begin{equation*}
\Ga(b)=\{1\}\cup\{a,a^2,\ldots,a^{n-1}\}\cup\{a^{i_1}b,a^{-i_1}b,\ldots,a^{i_t}b, a^{-i_t}b\}.
\end{equation*}
This implies that
\begin{equation*}
\Ga(b)\cap S=\{a^{i_1},a^{-i_1},\ldots,a^{i_t}, a^{-i_t}\}\cup\{a^{i_1}b,a^{-i_1}b,\ldots,a^{i_t}b, a^{-i_t}b\}.
\end{equation*}
If $\Delta=\la a\ra\setminus\{1\}$, then $\Ga\cong\K_{2n}$, as part (iii).

Let $\Delta\subset\la a\ra\setminus\{1\}$. Then $\Ga$ is non-complete. Since $\Ga(b)=\{1\}\cup(\Ga(b)\cap S)\cup(\Ga(b)\cap\Ga_2(1))$, it follows that
\begin{equation*}
|S|=n+2t-\xi,~|\Ga(b)\cap S|=4t-2\xi,\text{~and~} |\Ga(b)\cap \Ga_2(1)|=n-2t-1+\xi,
\end{equation*}
where $\xi = 1$ if $n = 2k$ and $a^k \in S$; otherwise, $\xi = 0$. Since $\Ga$ is arc-transitive, we have $|\Ga(x)\cap \Ga_2(1)|=|\Ga(b)\cap \Ga_2(1)|$ for all $x\in S$. Notice that
\[
|G\setminus(\{1\}\cup S)|=2n-1-(n+2t-\xi)=n-2t-1+\xi.
\]
It follows that $\Ga_2(1)=\Ga(x)\cap \Ga_2(1)$ for each $x\in S$, that is to say, each vertex in $S$ is adjacent to all vertices of $\Ga_2(1)$. In particular, $G=\{1\}\cup S\cup\Ga_2(1)$, and the subgraph $[\Ga_2(1)]$ is an edgeless graph. Let $H=\{1\}\cup \Ga_2(1)$. Then $S=G\setminus H$ and $H\subset \la a\ra$. Since $\Ga$ is non-complete, we have $|H|\geq2$. For each $h,g\in H$, since $\tau\in A_1$, we have $h^{-1}\in H$ and $gh^{-1}\in S\cup H$. If $gh^{-1}\in S$, then $h\neq g$ and $\{h,g\}$ is an edge of $\Ga$, contradicting to the subgraph $[\Ga_2(1)]$ is an edgeless graph. Thus, $gh^{-1}\in H$. It follows that $H$ is a unique subgroup of $\la a\ra$ and $H\neq \la a\ra$, which implies that $H\unlhd G\cong\D_{2n}$ and $|G:H|\geq3$. By Corollary~\ref{dual:kmb}, $\Ga\cong\K_{m[t]}$, with $mt=2n$, $t=|H|$ and $m\geq3$. Thus, part (iv) of this lemma holds.

\medskip
\noindent{\bf Case 2:} Assume that $a^\pi b\in S$ and $a^{1-\pi} b\notin S$ with $\pi\in\{0,1\}$.
\medskip

Since $\Ga$ be a connected inner-automorphic Cayley graph, then Propositions~\ref{dual} and~\ref{conju-D2n} assert that $n=2k$ is even and $a^k$ is the involution of the center of $G$. Moreover, again by Proposition~\ref{dual}, we have
\begin{equation*}
(a^\pi b)^G\subseteq S,~S\cap (a^{1-\pi}b)^G=\emptyset \text{~and~} \Inn(G)\leq A_1.
\end{equation*}
Note that $\la (a^\pi b)^G\ra=\la a^2,a^\pi b\ra\cong\D_{2k}$. Since $\Ga$ is connected, we have $\la S\ra=G$, and hence $S\cap\la a^2\ra a\neq\emptyset$. Moreover, we have $S=S_\pi$, where $S_\pi=(a^\pi b)^G\cup \Delta$ and $\Delta\subseteq\la a\ra$.
By Proposition~\ref{Aut:D2n}, we have $\theta_a\in\Aut(G)$, where $\theta_\a$ is induced by the mapping $a\mapsto a, b\mapsto ab$. It follows that
\[
\Cay(G,S_0)^{\theta_a}=\Cay(G,b^G\cup \Delta)^{\theta_a}=\Cay(G,(ab)^G\cup \Delta)=\Cay(G,S_1).
\]
Therefore, we may assume that $\pi=1$ in the remainder of the proof.

In this paragraph, we prove by contradiction that $S\cap\la a^2\ra=\emptyset$. Suppose that $S\cap\la a^2\ra\neq\emptyset$.
By Propositions~\ref{dual} and~\ref{conju-D2n}, we may assume that $S=(ab)^G\cup\Delta_0\cup\Delta_1$ with
\begin{align*}
\Delta_0=\{a^{2j_1},a^{-2j_1},\ldots,a^{2j_s},a^{-2j_s}\}\text{~and~}\Delta_1=\{a^{2i_1+1},a^{-2i_1-1},\ldots,a^{2i_\ell+1},a^{-2i_\ell-1}\},
\end{align*}
where $s$ and $\ell$ are two integers such that $1\leq s\leq \lceil (k-1)/2\rceil$ and $1\leq \ell\leq \lceil k/2\rceil$. For each element $g\in G$ and a subset $\Delta$ of $G$, we define $\Delta g=\{xg\mid x\in\Delta\}$. Moreover, a direct computation may easily verify that
\begin{align*}
\Ga(ab)=\la a^2\ra\cup\Delta_0ab\cup\Delta_1ab \text{~and~}
\Ga(a^{2j_1})=(ab)^G\cup\Delta_0a^{2j_1}\cup\Delta_1a^{2j_1};
\end{align*}
and for each $1\leq m\leq \ell$, we have $\Ga(a^{2i_m+1})=b^G\cup\Delta_0a^{2i_m+1}\cup\Delta_1a^{2i_m+1}$, where
\begin{align*}
\Delta_0ab\subseteq (ab)^G,
\Delta_1ab\subseteq b^G,
\Delta_0a^{2i_m+1}\subseteq\la a^2\ra a,
\Delta_1a^{2i_m+1}\subseteq\la a^2\ra,
\Delta_0a^{2j_1}\subseteq \la a^2\ra,
\Delta_1a^{2j_1}\subseteq \la a^2\ra a.
\end{align*}
Then we deduce that
\begin{align*}
&\Ga(ab)\cap S=(\la a^2\ra\cap \Delta_0)\cup(\Delta_0ab\cap (ab)^G)=\Delta_0\cup \Delta_0ab,  \\
&\Ga(a^{2i_m+1})\cap S=(\Delta_0a^{2i_m+1}\cap \Delta_1)\cup(\Delta_1a^{2i_m+1}\cap \Delta_0)\subseteq \Delta_0a^{2i_m+1}\cup  \Delta_0,\\
&\Ga(a^{2j_1})\cap S=(ab)^G\cup(\Delta_0a^{2j_1}\cap \Delta_0)\cup(\Delta_1a^{2j_1}\cap \Delta_1).
\end{align*}
Since $\Ga$ is arc-transitive, we have $|\Ga(ab)\cap S|=|\Ga(a^{2i_m+1})\cap S|$, and hence $\Delta_0a^{2i_m+1}\cap \Delta_1=\Delta_0a^{2i_m+1}$ and $\Delta_1a^{2i_m+1}\cap \Delta_0=\Delta_0$. It follows that $\Delta_0a^{2i_m+1}\subseteq \Delta_1$ and $\Delta_0\subseteq \Delta_1a^{2i_m+1}$, and so $\Delta_0\subseteq \Delta_1a^{2i_m+1}\cap \Delta_1a^{-2i_m-1}$. Since $a^{\pm 2i_m\pm1}$ is adjacent to all vertices in $\Delta_1a^{\pm2i_m\pm1}$ for every $1\leq m\leq \ell$, we conclude that each vertex in $\Delta_0$ is adjacent to each vertex in $\Delta_1$. Therefore, $\Delta_1\cup (ab)^G\subseteq \Ga(a^{2j_1})\cap S$, and then $|\Ga(a^{2j_1})\cap S|\geq k+2\ell$. Since $1\in \Delta_1a^{2i_m+1}\setminus \Delta_0$, it follows that $2\ell\geq 2s+1$. Again by the arc-transitivity of $\Ga$, we have $|\Ga(ab)\cap S|=|\Ga(a^{2j_1})\cap S|$, which means that $4s\geq k+2\ell\geq k+2s+1$, namely, $s\geq (k+1)/2$, leading to a contradiction. Therefore, $S\cap\la a^2\ra=\emptyset$.

The claim as above means that $S\subseteq (ab)^G\cup\la a^2\ra a$, and hence $|S|\leq 2k$. If $|S|=2k=n$, then $S=(ab)^G\cup\la a^2\ra a=G\setminus H$ with $H=\la a^2, ab\ra$, and so $\Ga\cong\K_{n,n}$ by Lemma~\ref{arctran:kmb}, as part (i). If $|S|=2k-1$, then since $(ab)^G\subseteq S$, there is a unique vertex $a^{r}\in \la a^2\ra a$ such that $a^r\notin S$. Thus, $a^{-r}\notin S$, which means that $a^r$ is an involution. Then $a^r=a^k\in Z(G)$ and $k\geq 3$ is odd. It follows that $S=(ab)^G\cup\la a^2\ra a\setminus\{a^k\}$. Thus, we derive from Lemma~\ref{dual:knn-nk2} that $\Ga\cong\K_{n,n}-n\K_2$, as part (ii). If $|S|\leq 2k-2$, then it follows from $S\cap\la a^2\ra a\neq\emptyset$ and $(ab)^G\subseteq S$ that $0<|S\cap \la a^2\ra a|\leq k-2$, as part (v).
This completes the proof.
\qed

The proof of Lemma~\ref{arctran:dualredu} yields the following corollary, which will be used repeatedly in the subsequent analysis.

\begin{corollary} \label{coro:isomorphic}
If $n$ is even and $(a^\pi b)^G\subseteq S_\pi\subseteq (a^\pi b)^G\cup\la a^2\ra a$ with $\pi\in\{0,1\}$,
then $\Cay(G,S_0)\cong\Cay(G,S_1)$.
\end{corollary}

It therefore remains to examine the properties of graphs satisfying part (v) of Lemma~\ref{arctran:dualredu}, which will complete the proof of Theorem~\ref{Thm:dihedrant}.

\begin{lemma} \label{pro:caseV}
Let $n=2k$ and let $(a^\pi b)^G\subseteq S\subseteq(a^\pi b)^G\cup\la a^2\ra a$ with $0<|\la a^2\ra a\cap S|\leq k-2$,
where $\pi\in\{0,1\}$. Then $\Ga=\Cay(G,S)$ has the following properties:
\begin{enumerate}[\rm (i)]
  \item $\Ga$ has girth $4$ and diameter $3$;
  \item $\Ga_2(1)=((a^{1-\pi}b)^G\cup\la a^2\ra)\setminus\{1\}$ and $\Ga_3(1)=\la a^2\ra a\setminus S$;
  \item $\Ga$ is a bipartite graph with biparts $\la a^2,b\ra$ and $\la a^2,b\ra a$.
\end{enumerate}
\end{lemma}

\proof By Corollary~\ref{coro:isomorphic}, we may assume $\pi=1$. Then we write
\begin{equation*}
S=(ab)^G\cup\{a^{2i_1+1},a^{-2i_1-1},\ldots,a^{2i_t+1},a^{-2i_s-1}\}
\end{equation*}
for an integer $1\leq s\leq \lceil(k-2)/2\rceil$. Then $(1,ab,a^2,a^3b)$ is a cycle of length $4$, and
\begin{align}
\nonumber&\Ga(ab)=\la a^2\ra\cup\{a^{2i_1+2}b,a^{-2i_1}b,\ldots,a^{2i_s+2}b,a^{-2i_s}b\},\\
&\Ga(a^{2i+1})=b^G \cup\{a^{2(i_1+i+1)},a^{2(i-i_1)},\ldots,a^{2(i+i_s+1)},a^{2(i-i_s)}\}, \label{eq:Gamma31}
\end{align}
where $1\leq i\leq k$. This implies that
\begin{align}
|\Ga(ab)\cap S|=0 \text{ and }|\Ga(a^{2i+1})\cap S|=0. \label{eq2:Gamma31}
\end{align}
Thus, $\Ga$ contains no triangles, and so $\Ga$ has girth $4$. For each $1\leq j\leq k$, we see that $\{a^{2j},ab\}$ and $\{a^{2j}b,a^{2i_1+1}\}$ are edges of $\Ga$. It follows that $(b^G\cup\la a^2\ra)\setminus\{1\}\subseteq\Ga_2(1)$, which implies that $\Ga(a^{2i+1})\subseteq\Ga_2(1)$. Since $|\la a^2\ra a\cap S|\neq 0$, we have $\la a^2\ra a\setminus S\neq\emptyset$. For each $x\in\la a^2\ra a\setminus S$, we have $x=a^{2\ell+1}$ for some integer $\ell$. By Eq.~\eqref{eq:Gamma31}, we obtain $|\Ga(x)\cap S|=0$, and so $x\in\Ga_3(1)$. Then we conclude that $\la a^2\ra a\setminus S\subseteq\Ga_3(1)$. Since $V(\Ga)=b^G\cup (ab)^G\cup\la a^2\ra\cup\la a^2\ra a$, we have $\Ga_2(1)=b^G\cup\la a^2\ra\setminus\{1\}$ and $\Ga_3(1)=\la a^2\ra a\setminus S$. Thus, $\Ga$ has diameter $3$, and so parts (i) and (ii) hold.

Recall that $\Ga$ is inner-automorphic. Then we derive from Eq.~\eqref{eq2:Gamma31} that $|\Ga(u)\cap S|=0$ for all $u\in S$. Thus, the induced subgraph $[S]$ is an edgeless graph. Moreover, for each $a^{2i+1}\in\la a^2\ra a$, by Eq.~\eqref{eq:Gamma31}, we have $|\Ga(a^{2i+1})\cap \Ga_3(1)|=0$. In particular, the induced subgraph $[\Ga_3(1)]$ is also an edgeless graph. Since
\[
\Ga(a^{2i}b)=\la a^2\ra a\cup\{a^{2i_1+2i+1}b,a^{-2i_1+2i-1}b,\ldots,a^{2i_s+2i+1}b,a^{-2i_s+2i-1}b\}
\]
and
\[
\Ga(a^{2i})=(ab)^G\cup\{a^{2i_1+2i+1},a^{-2i_1+2i-1},\ldots,a^{2i_s+2i+1},a^{-2i_s+2i-1}\},
\]
Since $\Ga_2(1)=b^G\cup\la a^2\ra\setminus\{1\}$ by (ii), we have $\Ga(a^{2i}b)\cap\Ga_2(1)=\emptyset$ and $\Ga(a^{2i})\cap\Ga_2(1)=\emptyset$ for all $1\leq i\leq k$. Hence $|\Ga(v)\cap\Ga_2(1)|=0$ for every $v\in\Ga_2(1)$. Thus, the induced subgraph $[\Ga_2(1)]$ is an edgeless graph. It follows that $\Ga$ has no cycles of length odd, and so $\Ga$ is a bipartite graph with biparts $\la a^2,b\ra$ and $\la a^2,b\ra a$. This completes the proof of part (iii).
\qed

We are ready to proof of Theorem~\ref{Thm:dihedrant} and Corollary~\ref{class:2distran}.

\medskip
\noindent{\bf Proof of Theorem~\ref{Thm:dihedrant}:} By Lemmas~\ref{arctran:dualredu} and~\ref{pro:caseV}, we have Theorem~\ref{Thm:dihedrant} holds. \qed

\medskip
\noindent{\bf Proof of Corollary~\ref{class:2distran}:}
Note that all graphs listed in cases (i)--(iv) of Theorem~\ref{Thm:dihedrant} are $2$-distance-transitive inner-automorphic Cayley graphs. We now assume that $\Ga$ is $2$-distance-transitive. It only need to show that case (v) of Theorem~\ref{Thm:dihedrant} cannot occur.

Assume that $\Ga=\Cay(G,S)$ satisfies case (v) of Theorem~\ref{Thm:dihedrant}. Then $\Ga$ is a bipartite graph of girth $4$ and diameter $3$. Moreover, the connection set $S$ satisfies $(a^\pi b)^G\subseteq S\subseteq(a^\pi b)^G\cup\la a^2\ra a$ with $0<|\la a^2\ra a\cap S|\leq k-2$, where $\pi\in\{0,1\}$ and $n=2k$. Let $|S\cap \la a^2\ra a|=t$. Then $t\leq  k-2$ and $|S|=k+t$. By Corollary~\ref{coro:isomorphic}, we may choose $\pi=1$. Then $(ab)^G\subseteq S\subseteq(ab)^G\cup\la a^2\ra a$.  Moreover, by Lemma~\ref{pro:caseV}, we have
\begin{equation*}
\Ga_2(1)=b^G\cup\la a^2\ra\setminus\{1\} \text{~and~} \Ga_3(1)=\la a^2\ra a\setminus S.
\end{equation*}
In particular, $|\Ga_2(1)|=2k-1$ and $|\Ga_3(1)|=k-t$. Note from Proposition~\ref{conju-D2n} that $\la a^2\ra a \subseteq\Ga(b)\subseteq \la a^2\ra a\cup (ab)^G$. It follows that $|\Ga(b)\cap\Ga_3(1)|=k-t$, and so $|\Ga(b)\cap S|=2t$. Since $\Ga$ is $2$-distance-transitive, we have $|\Ga(u)\cap S|=|\Ga(b)\cap S|=2t$ for all $u\in\Ga_2(1)$. As $\Ga$ has girth $4$, the number of edges between $S$ and $\Ga_2(1)$ is
\begin{align*}
(k+t)(k+t-1)=|S|\cdot(|S|-1)=|\Ga_2(1)|\cdot|\Ga(u)\cap S|=(2k-1)2t,
\end{align*}
which implies that $t^2-(2k-1)t+k^2-k=(t-k)(t-k+1)=0$. Therefore, $t=k$ or $t=k-1$, contradicting to $t\leq k-2$. This completes the proof.\qed

\section{Construction of arc-transitive inner-automorphic Cayley graphs}\label{sec:structure}

In this section, we will prove Theorem~\ref{Thm:dualCay} by a series of lemmas.
Firstly, we present the following hypothesis.

\begin{hypothesis}\label{hypothesis}
Let $G=\la a,b\mid a^{4p}=b^2=1,a^b=a^{-1}\ra\cong\D_{8p}$ for an odd prime $p$, and let $O=\{a^i\mid (i,4p)=1\}$ be the set of all generators of $\la a\ra$. Define
\begin{align*}
\Ga=\Cay(G,S), \text{~where~} S=(a b)^G\cup O.
\end{align*}
Let $A=\Aut(\Ga)$. Let $\Sig$ be the quotient graph of $\Ga_{\la R(a^{2p})\ra}$. Let $\BB=\BB_1\cup\BB_2$ be the set of all $\la R(a^{2p})\ra$-orbits on $G$, where
\begin{align*}
&\BB_1:=\{\{a^{2i+1}b,a^{2p+2i+1}b\},\{a^{2i+1},a^{2p+2i+1}\}\mid 1\leq i\leq p\},\\
&\BB_2:=\{\{a^{2i}b,a^{2p+2i}b\},\{a^{2i},a^{2p+2i}\}\mid 1\leq i\leq p\}.
\end{align*}
\end{hypothesis}

From Lemma~\ref{pro:caseV}, we know that
\begin{align}\label{graph}
S=(ab)^G\cup \la a^2\ra a\setminus\{a^p,a^{3p}\},~\Ga_2(1)=b^G\cup\la a^2\ra\setminus\{1\}, \text{~and~} \Ga_3(1)=\{a^p,a^{3p}\}.
\end{align}
Then $\Ga$ has valency $4p-2$ and exhibits four distinct types of edges, which are as follows:
\begin{align}
&u_{i,j}:=\{a^{2i},a^{2j+1}b\}\in E(\Ga) \text{~for all~} 1\leq i,j\leq 2p; \label{uij}\\
&v_{i,j}:=\{a^{2i},a^{2j+1}\}\in E(\Ga) \text{~for all~} 1\leq i,j\leq 2p \text{~and~} a^{2j-2i+1}\notin\{a^p,a^{3p}\};\label{vij}\\
&w_{i,j}:=\{a^{2i+1},a^{2j}b\}\in E(\Ga) \text{~for all~} 1\leq i,j\leq 2p;\label{wij}\\
&z_{i,j}:=\{a^{2i}b,a^{2j+1}b\}\in E(\Ga) \text{~for all~} 1\leq i,j\leq 2p \text{~and~} a^{2j-2i+1}\notin\{a^p,a^{3p}\}.\label{zij}
\end{align}
Moreover, we have
\begin{align}\label{neigh}
\Ga(a^p)=\Ga(a^{3p})=\Ga_2(1)\setminus\{a^{2p}\} \text{~and~} \Ga(1)=\Ga(a^{2p})=S.
\end{align}
In what follows, we use the notations as above. Now, we study the properties of $\Ga$ and $\Sig$.

\begin{lemma}\label{quotitent}
Suppose Hypothesis~$\ref{hypothesis}$. Then the following holds.
\begin{enumerate}[\rm (i)]
  \item For each $x\in G$, there is a unique vertex $y\in G$ such that $\Ga(x)=\Ga(y)$.
  \item For each $B:=\{x,y\}\in \BB$, the induced subgraph $[B]$ of $\Ga$ is an edgeless graph, and $\Ga(x)=\Ga(y)$.
  \item For each $B\in \BB_1$ and $C\in\BB_2$, either $[B\cup C]$ is an edgeless graph, or $[B\cup C]\cong\C_4$.
       In particular, $\Ga$ is a normal $2$-cover of $\Sig$.
  \item $\Sig\cong\K_{2p,2p}-2p\K_2$ is a bipartite graph with biparts $\BB_1$ and $\BB_2$.
\end{enumerate}
\end{lemma}

\proof By Eq.~\eqref{neigh} we know that $\Ga(1)=\Ga(a^{2p})=S$. Let $g\in G\setminus\{1,a^{2p}\}$ be such that $\Ga(g)=S$. By Lemma~\ref{pro:caseV}, $\Ga$ is a bipartite graph and $\Ga_3(1)=\{a^p,a^{3p}\}$, which implies that $g\in \Ga_2(1)$ and $\Ga(a^p)\subseteq \Ga_2(1)$. Note that $|\Ga_2(1)\setminus\{a^{2p},g\}|=4p-3$ and $|S|=4p-2$. Thus, $|\Ga(a^p)\cap\{a^{2p},g\}|\neq\emptyset$, that is, $a^p$ is adjacent to $a^{2p}$ or $g$, contradicting to $S=(ab)^G\cup \la a^2\ra a\setminus\{a^p,a^{3p}\}$ and $\Ga(g)=S$. Consequently, for the vertex $1\in G$, there exists precisely one other vertex, namely $a^{2p}$, such that $\Ga(1)=\Ga(a^{2p})$. By the vertex-transitivity of $\Ga$, part (i) holds.

Recall that $S=(ab)^G\cup O$ with $O$ the set of all generators of $\la a\ra$. Let $B_1=\{1,a^{2p}\}$. Then $B_1\in\BB$, and $[B_1]$ is an edgeless graph. By Eq.~\eqref{neigh}, we have $\Ga(1)=\Ga(a^{2p})$. Since $\la R(a^{2p})\ra\unlhd R(G)$, it follows that $\BB$ is a complete imprimitive block system of $R(G)$ on $V(\Ga)$. Moreover, $R(G)$ acting on $\BB$ induces a transitive permutation group, and so the quotient graph $\Sig$ is vertex-transitive. Thus, for all $B:=\{x,y\}\in \BB$, the induced subgraph $[B]$ of $\Ga$ is an edgeless graph and $\Ga(x)=\Ga(y)$, as part (ii).

Let $B\in\BB_1$ and let $C\in\BB_2$. Assume that $[B\cup C]$ is not an edgeless graph. By part (ii), $[B\cup C]$ contains an edge $\{x,y\}$ with $x\in B$ and $y\in C$. This implies that $xa^{2p}$ and $ya^{2p}$ is also an edge of $\Ga$. Note that $B=\{a^{2i+1}b,a^{2p+2i+1}b\}$ or $\{a^{2i+1},a^{2p+2i+1}\}$ for some integer $i$ with $1\leq i\leq p$. Let $x=a^{2i+1}b\in B$. Then $y=a^{2s-2i}$ for some $1\leq s\leq 2p$, or $y=a^{2t+2i+2}b$ for some $1\leq t\leq 2p$ and $t\notin\{(p-1)/2,(3p-1)/2\}$. It follows that $(x,y,xa^{2pa},ya^{2p})$ is a cycle of $\Ga$, and hence $[B\cup C]\cong\C_4$. Similarly, we have $[B\cup C]\cong\C_4$ if $x=a^{2i+1}$. In particular, $\Ga$ is a normal $2$-cover of $\Sig$, as part (iii).

Note that the induced subgraphs $[\BB_1]$ and $[\BB_2]$ in $\Sig$ are edgeless graphs. Thus, $\Sig$ is a bipartite graph with biparts $\BB_1$ and $\BB_2$. Moreover, $B_1=\{1,a^{2p}\}$ is adjacent to all elements of $\BB_1\setminus\{\{a^{p},a^{3p}\}\}$ in $\Sig$. This implies that $\Sig$ has order $4p$ and valency $2p-1$. Therefore, $\Sig\cong\K_{2p,2p}-2p\K_2$, as part (iv).\qed

Let $\a_t, \beta_t,\gamma_t$ and $\delta_t$, with $1\leq t\leq p$, be permutations of $\Sym(G)$ as follows:
\begin{align*}
&\a_t=(a^{2t+1}b,a^{2p+2t+1}b), ~\beta_{t}=(a^{2t+1},a^{2p+2t+1}),\\
&\gamma_t=(a^{2t}b,a^{2p+2t}b),~\delta_t=(a^{2t},a^{2p+2t}).
\end{align*}
Let $K=\la\a_t, \beta_t,\gamma_t,\delta_t\mid 1\leq t\leq p\ra$. Then $K\cong\ZZ_2^{4p}$. Next, we investigate the automorphism groups of $\Ga$ and $\Sig$.

\begin{lemma}\label{kernel}
Suppose Hypothesis~$\ref{hypothesis}$. Then the following holds.
\begin{enumerate}[\rm (i)]
  \item $K\cong\ZZ_2^{4p}$ is the kernel of $A$ acting on $\BB$, and in particular, $A/K\leq\Aut(\Sig)$.
  \item $R(a^{2p})=\prod_{t=1}^{p}\a_t\beta_t\gamma_t\delta_t$, and $\la R(a^{2p})\ra$ is the kernel of $R(G)$ acting on $\BB$.
  \item $R(G)/\la R(a^{2p})\ra\cong\D_{4p}$ is regular on $\BB$, and $R(\la a^2,b\ra)/\la R(a^{2p})\ra\cong\D_{2p}$ is regular on $\BB_i$ for $i\in\{1,2\}$.
\end{enumerate}
\end{lemma}

\proof Recall that Eq.~\eqref{uij}--\eqref{zij}. For each $1\leq t\leq p$, we have that $\a_t$ fixes $v_{i,j}$ for all $1\leq i,j\leq 2p$ and $a^{2j-2i+1}\notin\{a^p,a^{3p}\}$, and also fixes $w_{i,j}$ for all $1\leq i,j\leq 2p$. Moreover, we obtain that
\[
u_{i,j}^{\a_t}=\begin{cases}
u_{i,j},&\text{ if }1\leq i,j\leq 2p \text{ and } j\notin\{t,p+t\}\\
\{a^{2i},a^{2p+2t+1}b\},&\text{ if }1\leq i,j\leq 2p \text{ and } j=t\\
\{a^{2i},a^{2t+1}b\},&\text{ if }1\leq i,j\leq 2p \text{ and } j=p+t,\\
\end{cases}
\]
and
\[
z_{i,j}^{\a_t}=\begin{cases}
z_{i,j},&\text{ if }1\leq i,j\leq 2p \text{ and }j\notin\{t,p+t\}\\
\{a^{2i}b,a^{2p+2t+1}b\},&\text{ if }1\leq i,j\leq 2p \text{ and }j=t\\
\{a^{2i}b,a^{2t+1}b\},&\text{ if }1\leq i,j\leq 2p \text{ and }j=p+t.\\
\end{cases}
\]
It follows that $u_{i,j}^{\a_t},z_{i,j}^{\a_t}\in E(\Ga)$. Then we derive that $\a_t\in A$ for all $1\leq t\leq p$. By a similar argument, we conclude that $\beta_t,\gamma_t,\delta_t\in A$ for every $1\leq t\leq p$. Therefore, $K\leq A$.

Let $L$ be the kernel of $A$ acting on $\BB$. Note that $K$ fixes $\BB$ pointwise, and hence $K\leq L$. Conversely, let $\tau\in L$. Let $\BB=\{B_1,B_2,\ldots,B_{4p}\}$. If $\tau$  interchanges the two elements in $B_i$ for every $1\leq i\leq 4p$, then $\tau=\prod_{t=1}^{p}\a_t\beta_t\gamma_t\delta_t\in K$. Thus, we may suppose that $\tau$ fixes $B_i$ pointwise for some $B_i\in\BB$. Without loss of generality, assume that $\tau$ fixes $B_i$ pointwise for all $1\leq i\leq \ell$ and interchanges the two elements in $B_i$ for every $\ell+1\leq i\leq 4p$. If $\ell=4p$, then $\tau$ fixes $G$ pointwise, and so $\tau=1\in K$. If $\ell<4p$, then $\tau$ is in fact of the product of some $\a_t, \beta_t,\gamma_t$ and $\delta_t$, where $1\leq t\leq p$. Thus, $\tau\in K$, which implies that $K=L$. Furthermore, $A/K\leq\Aut(\Sig)$, as part (i).

To prove part (ii), let $\mu=\prod_{t=1}^{p}\a_t\beta_t\gamma_t\delta_t$. Then for each $x\in G$, we have $x^\mu=xa^{2p}=x^{R(a^{2p})}$, and so $R(a^{2p})=\mu$. By part (i), $K\cong\ZZ_2^{4p}$ is the kernel of $A$ acting on $\BB$,
and so $K\cap R(G)$ is the kernel of $R(G)$ acting on $\BB$. Then $K\cap R(G)$ is a normal $2$-subgroup of $R(G)\cong\D_{8p}$.
Therefore, $K\cap R(G)=\la R(a^{2p})\ra$, which means that $\la R(a^{2p})\ra$ is the kernel of $R(G)$ acting on $\BB$.

Since $R(G)$ is regular on $G$, we have $R(G)/\la R(a^{2p})\ra$ is regular on $\BB$. Recall that $\BB_1=\{\{a^{2i+1}b,a^{2p+2i+1}b\},\{a^{2i+1},a^{2p+2i+1}\}\mid 1\leq i\leq p\}$ and $\BB_2=\{\{a^{2i}b,a^{2p+2i}b\},\{a^{2i},a^{2p+2i}\}\mid 1\leq i\leq p\}$. Since $R(\la a^2,b\ra)$ acts regularly on both $\la a^2,b\ra$ and $\la a^2,b\ra a$, it follows that $R(\la a^2,b\ra)/\la R(a^{2p})\ra\cong\D_{2p}$ is regular on $\BB_i$ for $i=1,2$.
This completes the proof.\qed

By Lemma~\ref{quotitent}, we have $\Sig=\K_{2p,2p}-2p\K_2$. Let $\ov{A}=A/K$, and let $\ov{A}^+=\ov{A}_{\BB_1}=\ov{A}_{\BB_2}$. It is easy to see that, for each $x\in A_{\BB_1}=A_{\BB_2}$, the element $x$ fixes $\BB_1$ pointwise if and only if $x$ fixes $\BB_2$ pointwise. Thus, for  $i\in\{1,2\}$, the group $\ov{A}^+$ is faithful on $\BB_i$, and so we identity $\ov{A}^+$ with its restriction on $\BB_i$. Moreover, $|\ov{A}:\ov{A}^+|=2$, and $\ov{A}_B=A_{B}$ for each $B\in\BB$. It follows from Lemma~\ref{kernel} that $R(\la a^2,b\ra)/\la R(a^{2p})\ra\leq \ov{A}^+$ acts transitively on $\BB_i$. We now show that $\ov{A}^+$ is primitive on $\BB_i$.

\begin{lemma}\label{aut:group}
Suppose Hypothesis~$\ref{hypothesis}$ holds, and let $i\in\{1,2\}$. Then $\ov{A}^+$ is a primitive group on $\BB_i$ containing a regular subgroup $R(\la a^2,b\ra)K/K\cong\D_{2p}$. Moreover, for $B\in\BB$, the triple $(\ov{A}^+,\ov{A}^+_B,|\BB_i|)$ equals either $(\Sy_{2p},\Sy_{2p-1},2p)$ or $(\PGL(2,r^f).\ZZ_e,\ZZ_r^f\rtimes\ZZ_{r^f-1}.\ZZ_e,r^f+1)$ for some positive integers $f$ and $r$ with $r^f\equiv1\pmod{4}$ and $e\mid f$.
\end{lemma}

\proof Suppose to the contrary that $\ov{A}^+$ is imprimitive on $\BB_i$ for $i\in\{1,2\}$. Let $\Delta$ be a nontrivial block of $\ov{A}^+$ acting on $\BB_i$. Then $1<|\Delta|<2p$. By Lemma~\ref{kernel}, we know that
\begin{align*}
\ov{A}^+\geq R(\la a^2,b\ra)K/K\cong R(\la a^2,b\ra)/(K\cap R(\la a^2,b\ra))= R(\la a^2,b\ra)/\la R(\la a^{2p})\ra\cong\D_{2p}
\end{align*}
and $R(\la a^2,b\ra)K/K$ is regular on $\BB_i$. Then $R(\la a^2,b\ra)K/K$ acting on $\Delta$ must be regular, and hence $|\Delta|=|\BB_i|=2p$, leading to a contradiction. Thus, $\ov{A}^+$ is primitive on $\BB_i$ with $i\in\{1,2\}$, and contains a regular subgroup $R(\la a^2,b\ra)K/K\cong\D_{2p}$.

Note that all primitive permutation groups contains a regular dihedral subgroup are listed in Proposition~\ref{Qd-group}. Let $B\in\BB$. By Table~\ref{d-group}, one of the following holds:
\begin{enumerate}[\rm (I)]
  \item $|\BB_i|\in\{4,8,16,12,22,24,4n\}$, where $n\geq1$ is an integer;
  \item $|\BB_i|=r^f+1$, and $(\ov{A}^+,\ov{A}^+_B)=(\PSL(2,r^f).o,\ZZ_r^f\rtimes\ZZ_{(r^f-1)/2}.o)$ with $r^f\equiv3\pmod{4}$ and $o\leq\ZZ_2\times\ZZ_f$, or $(\PGL(2,r^f).\ZZ_e,\ZZ_r^f\rtimes\ZZ_{r^f-1}.\ZZ_e)$ with $r^f\equiv1\pmod{4}$ and $e\mid f$;
  \item $\ov{A}^+=\Sy_{2n}$, and $\ov{A}^+_B=\Sy_{2n-1}$ with $|\BB_i|=2n$ for a positive integer $n$.
\end{enumerate}
Since $|\BB_i|=2p$ with $p$ an odd prime, we conclude that part (I) cannot occur, and $n=p$ for part (III). If $r^f\equiv3\pmod{4}$, then $4$ divides $r^f+1=2p$, a contradiction. Therefore, we have part (II) with $r^f\equiv1\pmod{4}$ and part (III) with $n=p$ hold, completing the proof.\qed

By Proposition~\ref{Aut:D2n}, we know that $\Aut(G)=\la \theta_{a^i},\tau_{a^j}\mid 1\leq i,j\leq n, (j,4p)=1\ra=\la\theta_{a}\ra\rtimes\la \tau_{a^j}\mid 1\leq j\leq n, (j,4p)=1\ra$,
where
\begin{align*}
&\theta_{a^i}:a\mapsto a,~b\mapsto a^ib, \text{~where~} 1\leq i\leq 4p;\\
&\tau_{a^j}:a\mapsto a^j,~b\mapsto b,\text{~where~} 1\leq j\leq 4p \text{~and~} (j,4p)=1.
\end{align*}
We now determine the full automorphism group $A$, and show that $\Ga$ is arc-transitive.

\begin{lemma}\label{autom}
Suppose Hypothesis~$\ref{hypothesis}$. Let $B=\{1,a^{2p}\}\in\BB$. Then the following holds.
\begin{enumerate}[\rm (i)]
  \item $\Aut(G,S)=\la \theta_{a^{2i}},\tau_{a^j}\mid 1\leq i,j\leq 4p,(j,4p)=1\ra$ is a subgroup of $\Aut(G)$ with index $2$.
  \item $\ov{A}^+\cong\Sy_{2p}$, $\ov{A}\cong\Sy_{2p}\times\ZZ_2$, and $\ov{A}_B\cong\Sy_{2p-1}$.
  \item $A\cong\ZZ_2^{4p}.(\Sy_{2p}\times\ZZ_2)$ and $A_1\cong\ZZ_2^{4p-1}.\Sy_{2p-1}$.
  \item $\Ga$ is arc-transitive.
\end{enumerate}
\end{lemma}

\proof Let $H= \la \theta_{a^{2i}},\tau_{a^j}\mid 1\leq i,j\leq 4p,(j,4p)=1\ra$. Then $H$ is a subgroup of $\Aut(G)$ with index $2$. Recall that $S=(ab)^G\cup O$ with $O$ the set of all generators of $\la a\ra$. Then for each $\a\in\Aut(G)$, $\a$ fixes $S$ setwise if and only if $\a\in H$. Therefore, $\Aut(G,S)=H$, as part (i).

Note that $\Aut(G,S)\leq A_1$ and $\Aut(G,S)$ fixes $B$ setwise. Recall the definition of $K$ above Lemma~\ref{kernel}. Clearly, $\Aut(G,S)\cap K=1$. It follows that $\Aut(G,S)\cong\Aut(G,S)/(\Aut(G,S)\cap K)\cong\Aut(G,S)K/K\leq \ov{A}_B=\ov{A}^+_B$. By Lemma~\ref{aut:group}, we obtain
$(\ov{A}^+,\ov{A}^+_B,|\BB_i|)=(\PGL(2,r^f).\ZZ_e,\ZZ_r^f\rtimes\ZZ_{r^f-1}.\ZZ_e,r^f+1)$, or $(\Sy_{2p},\Sy_{2p-1},2p)$. If $\ov{A}^+_B=\ZZ_r^f\rtimes\ZZ_{r^f-1}.\ZZ_e$, then $2p=r^f+1$, and hence $(2p,r^fe)=1$. Since $\Aut(G,S)\leq \ov{A}^+_B$, we have that $2p$ divides $|\ov{A}^+_B|$. This implies that $2p$ is a divisor of $r^f-1$, which is clearly impossible. Therefore, $\ov{A}^+=\Sy_{2p}$ and $\ov{A}^+_B=\ov{A}_B=\Sy_{2p-1}$. Since $|\ov{A}:\ov{A}^+|=2$ and $\ov{A}\leq \Aut(\Sig)\cong\Sy_{2p}\times\ZZ_2$, we conclude that $\ov{A}\cong \Sy_{2p}\times\ZZ_2$, as part (ii).

By part (ii), we have $A/K=\ov{A}\cong \Sy_{2p}\times\ZZ_2$. Recall that  $K=\la\a_t, \beta_t,\gamma_t,\delta_t\mid 1\leq t\leq p\ra\cong\ZZ_2^{4p}$. Then $A\cong\ZZ_2^{4p}.(\Sy_{2p}\times\ZZ_2)$. Since $|A:A_1|=|G|=8p$, we have $|A_1|=2^{4p-1}\cdot(2p-1)!$. Notice that $K_1=\la\a_t, \beta_t,\gamma_t,\delta_r\mid 1\leq t,r\leq p,r\neq p\ra\cong\ZZ_2^{4p-1}$. It follows from $A_1K/K\cong A_1/(A_1\cap K)\cong A_1/K_1\leq \ov{A}_B\cong \Sy_{2p-1}$ that $A_1/K_1=\ov{A}_B$. Hence $A_1\cong\ZZ_2^{4p-1}.\Sy_{2p-1}$, as part (iii).

Let $L$ be the kernel of $A_1$ acting on $S$. Then $K_1\leq L$. It follows from $A_1/L\cong (A_1/K_1)/(L/K_1)\cong \Sy_{2p-1}/(L/K_1)$ that $L/K_1\in \{1,\A_{2p-1},\Sy_{2p-1}\}$. Since $\Aut(G,S)\leq A_1$ and $\Aut(G,S)$ acting on $S$ has two orbits $(ab)^G$ and $O$, we have that $A_1$ acting on $S$ has at most two orbits. In particular, $A_1$ acts on $S$ by changing at least $|O|=2p-2$ points. This implies that $L/K_1=1$, and so $L=K_1$. Thus, $A_1^S\cong\Sy_{2p-1}$. If $A_1^S$ is intransitive on $S$, then $(ab)^G$ and $O$ are all orbits of $A_1^S$ acting on $S$. It follows that $A_1^S$ acts on $(ab)^G$ induces a transitive permutation group of degree $2p$, and so $A_1^S\cong\Sy_{2p-1}$ has a subgroup of index $2p$, which is clearly impossible. Thus, $A_1$ is transitive on $S$, and so $\Ga$ is arc-transitive. \qed

We are ready to proof of Theorem~\ref{Thm:dualCay}.

\medskip
\noindent{\bf Proof of Theorem~\ref{Thm:dualCay}:}
Let $G=\la a,b\mid a^{4p}=b^2=1,a^b=a^{-1}\ra\cong\D_{8p}$ with $p$ an odd prime, and let $S_\pi=(a^\pi b)^G\cup \{a^i\mid (i,4p)=1\}$ with $\pi\in\{0,1\}$. By Propositions~\ref{dual} and~\ref{conju-D2n}, $\Cay(G,S_\pi)$ is a connected inner-automorphic Cayley graph. Moreover, $\Cay(G,S_0)\cong \Cay(G,S_1)$ by Corollary~\ref{coro:isomorphic}. It follows from Lemma~\ref{autom} that $\Cay(G,S_\pi)$ is arc-transitive and $\Aut(\Cay(G,S_\pi))\cong\ZZ_2^{4p}.(\Sy_{2p}\times\ZZ_2)$. This completes the proof.
\qed

At the end of this section, we provide further arc-transitive inner-automorphic Cayley graphs satisfying part (v) of Theorem~\ref{Thm:dihedrant}, which are non-isomorphic to the graphs in Theorem~\ref{Thm:dualCay}.

\begin{example}\label{D60}
Let $G=\la a,b\mid a^{30}=b^2=1,a^b=a^{-1}\ra\cong\D_{60}$, and define the following subsets of $G$:
\begin{align*}
O_1=\{a^5,a^{25}\},~O_2=\{a^3,a^9,a^{21},a^{27}\},~O_3=\{a,a^{7},a^{11},a^{13},a^{17},a^{19},a^{23},a^{29}\}.
\end{align*}
Then $O_1$, $O_2$ and $O_3$ are precisely the set of all elements of order $6$, $10$ and $30$ in $G$, respectively. For $\pi\in\{0,1\}$, write
\begin{align*}
S_\pi=(a^\pi b)^G \cup O_1 \cup O_3,~R_\pi=(a^\pi b)^G \cup O_2 \cup O_3.
\end{align*}
Then $|S_\pi| = 25$ and $|R_\pi| = 27$. Now let $\Ga_{25}^\pi=\Cay(G, S_\pi)$ and let $\Ga_{27}^\pi=\Cay(G, R_\pi)$. Using Magma~\cite{Magma}, we find that the following holds:
\begin{enumerate}[\rm (i)]
  \item $\Ga_{25}^\pi$ and $\Ga_{27}^\pi$ are connected arc-transitive inner-automorphic Cayley graphs;
  \item $\Ga_{25}^\pi$ and $\Ga_{27}^\pi$ are bipartite graphs of girth $4$ and diameter $3$;
  \item $|\Aut(\Ga_{25}^\pi)| = 2^{41}\cdot 3^{14}\cdot5^{13}$ and $|\Aut(\Ga_{27}^\pi)|= 2^{29}\cdot3^{24}\cdot5^2\cdot57$.
\end{enumerate}
\end{example}

\begin{example}\label{D84}
Let $G=\la a,b\mid a^{42}=b^2=1,a^b=a^{-1}\ra\cong\D_{84}$  and define
\begin{align*}
&O_1=\{a^7,a^{35}\},~O_2=\{a^3,a^9,a^{15},a^{27},a^{33},a^{39}\},~\text{and}\\
&O_3=\{a,a^5,a^{11},a^{13},a^{17},a^{19},a^{23},a^{25},a^{29},a^{31},a^{37},a^{41}\}.
\end{align*}
Then $O_1$, $O_2$ and $O_3$ are the set of all elements of order $6$, $14$ and $42$ in $G$, respectively.
For $\pi\in\{0,1\}$, write
\begin{align*}
S_\pi=(a^\pi b)^G \cup O_1 \cup O_3,~R_\pi=(a^\pi b)^G \cup O_2 \cup O_3.
\end{align*}
Then $|S_\pi| = 35$ and $|R_\pi| = 39$. Let $\Ga_{35}^\pi=\Cay(G, S_\pi)$ and let $\Ga_{39}^\pi=\Cay(G, R_\pi)$.
Then computing by Magma~\cite{Magma}, the following holds:
\begin{enumerate}[\rm (i)]
  \item $\Ga_{35}^\pi$ and $\Ga_{39}^\pi$ are connected arc-transitive inner-automorphic Cayley graphs;
  \item $\Ga_{35}^\pi$ and $\Ga_{39}^\pi$ are bipartite graphs of girth $4$ and diameter $3$;
  \item $|\Aut(\Ga_{35}^\pi)| = 2^{53}\cdot 3^{26} \cdot 5^{13} \cdot 7^{12}$ and $|\Aut(\Ga_{39}^\pi)|= 2^{40}\cdot 3^{33} \cdot 5^2 \cdot 7^2\cdot 11 \cdot13$.
\end{enumerate}
\end{example}

\section*{Acknowledgements}
%The authors would like to thank the anonymous referee for careful reading and valuable suggestions to this paper.
This work was partially supported by the National Natural Science Foundation of China (12501469).

\end{document}